\documentclass{amsproc}
\usepackage{euscript, graphicx, epstopdf}
\usepackage{cases}
\usepackage{mathrsfs}
\usepackage{bbm}
\usepackage{amssymb}
\usepackage{txfonts}
\usepackage{amscd}
\usepackage{amsfonts,latexsym,amsmath,amsxtra,mathdots,amssymb,latexsym,mathabx}
\usepackage[all,cmtip]{xy}
\usepackage{color}
\usepackage{colordvi}
\usepackage{multicol}
\usepackage{hyperref}
\usepackage{tikz}
\usepackage{float}
\usepackage{setspace, floatflt}
\usepackage{accents}

\allowdisplaybreaks

\newcommand{\BC}{{\mathbb {C}}}

 \newcommand{\BR}{{\mathbb {R}}}

 \newcommand{\BZ}{{\mathbb {Z}}}

\newcommand{\GL}{{\mathrm {GL}}} \newcommand{\PGL}{{\mathrm {PGL}}}
\newcommand{\SL}{{\mathrm {SL}}} \newcommand{\PSL}{{\mathrm {PSL}}}

 \newcommand{\Tr}{{\mathrm{Tr}}}

\newcommand{\ra}{\rightarrow}

\def\-{^{-1}}

\def\lp {\left (}
\def\rp {\right )}

\def\BCx{\BC^{\times}}
\def\Voronoi{Vorono\" \i \hskip 3 pt}

\renewcommand{\Re}{{\mathrm{Re}\,}}

\def\bfJ {\boldsymbol J}
\def\nwedge {\hskip - 2 pt \wedge \hskip - 2 pt }
\def\shskip{\hskip 0.5 pt}
\def\tw{\textit{w}}

\def\dx{d^{\times} \hskip -1 pt}

\makeatletter
\g@addto@macro\normalsize{\setlength\abovedisplayskip{3pt}}
\makeatother

\makeatletter
\g@addto@macro\normalsize{\setlength\belowdisplayskip{3pt}}
\makeatother

\newcommand{\delete}[1]{}

\theoremstyle{plain}

\newtheorem{thm}{Theorem}[section] \newtheorem{cor}[thm]{Corollary}
\newtheorem{lem}[thm]{Lemma}  \newtheorem{prop}[thm]{Proposition}

 \newtheorem{defn}[thm]{Definition}

\newtheorem {rem}[thm]{Remark}

\numberwithin{equation}{section}
\newtheorem*{acknowledgement}{Acknowledgements}

\begin{document}

	\title{Bessel Identities in the Waldspurger Correspondence over the Complex Numbers}

	\author{Jingsong Chai}
	\address{School of Mathematics\\ Sun Yat-sen University \\Guangzhou,  510275\\China}
	\email{chaijings@mail.sysu.edu.cn}

	\author{Zhi Qi}
	\address{School of Mathematical Sciences\\ Zhejiang University\\Hangzhou 310027\\China}
	\email{zhi.qi@zju.edu.cn}

	\thanks{The first author is supported by the National Natural Science
Foundation of China [Grant 11771131].}

	\begin{abstract}
		We prove certain identities between relative Bessel functions attached to irreducible unitary representations of $\mathrm{PGL}_2(\mathbb{C})$ and Bessel functions attached to irreducible unitary representations of $\mathrm{SL}_2 (\mathbb{C})$. These identities reflect the Waldspurger correspondence over $\mathbb{C}$. We also prove several regularity theorems for Bessel and relative Bessel distributions which appear in the relative trace formula. This paper constitutes the local   spectral theory of Jacquet's relative trace formula over $\mathbb{C}$.
	\end{abstract}
	
	\maketitle
	
	\section{Introduction}
	
	\subsection{Motivations}
	
	There is a pair of exponential integral formulae of Weber and Hardy on the Fourier transform of Bessel functions on the real numbers. Let $e (x) = e^{2 \pi i x}$.
	Weber's formula   is as follows,
	\begin{equation}\label{0eq: Weber's formula}
	\int_0^\infty \frac 1 {\sqrt x}  J_{\nu} \lp 4 \pi \sqrt x\rp e \lp {\pm x y} \rp   {d x}  = \frac 1 {\sqrt {2 y} } e \lp {\mp  \lp \frac 1 {2 y} - \frac 1 8  \nu - \frac 1 8 \rp} \rp J_{\frac 1 2 \nu} \lp \frac {\pi} y \rp,
	\end{equation}
	with $y \in (0, \infty)$, valid when  $\Re \nu > - 1$.  Hardy's formula  is in a similar fashion,
	\begin{equation}\label{0eq: Hardy's formula}
	\begin{split}
	\int_0^\infty \frac 1 {\sqrt x}  K_{\nu} \lp 4 \pi \sqrt x\rp e \lp {\pm x y} \rp   {d x} & =   - \frac {\pi} {2 \sin ( \pi \nu)}  \frac 1 {\sqrt {2 y} } e \lp {\pm  \lp \frac 1 {2 y} + \frac 1 8 \rp} \rp \\
	& \lp e \lp \pm \frac 1 8 \nu \rp J_{\frac 1 2 \nu} \lp \frac {\pi} y \rp -
	e \lp \mp \frac 1 8 \nu \rp J_{- \frac 1 2 \nu} \lp \frac {\pi} y \rp \rp,
	\end{split}
	\end{equation}
	when $|\Re \nu | < 1$. Here $J_{\nu}$ and $K_{\nu}$ are  Bessel functions (see \cite{Watson}).
	
	In the work  \cite{BaruchMao-Real} of Baruch and Mao, the formulae \eqref{0eq: Weber's formula} and \eqref{0eq: Hardy's formula} are used to establish an identity between the relative Bessel functions for $ \PGL_2 (\BR)  $ and the Bessel functions for $\widetilde {\SL}_2 (\BR)$ and hence a correspondence from  irreducible unitary representations of $\PGL_2 (\BR)$ to  irreducible genuine unitary representations of $\widetilde {\SL}_2 (\BR)$. This correspondence is exactly the Shimura-Waldspurger correspndence over $\BR$!
	Completely analogous results for the non-Archimedean case were  obtained in  \cite{BaruchMao-NA}. These results fit into the theory of the relative trace formula developed by
	Jacquet and constitute the local (real and non-Archimedean) spectral theory that complements the global theory in \cite{Jacquet-RTF}. Ultimately,  the Waldspurger formula over a totally real field was obtained and used to study the central
	value of $\PGL_2$ automorphic $L$-functions in \cite{BaruchMao-Global}.

	Recently, it is proven in \cite{Qi-Sph,Qi-II-G} the following complex analogue of the classical formulae of Weber and Hardy,
	\begin{equation}\label{eq: main 1}
	\begin{split}
	\int_{0}^{2 \pi} \int_0^\infty \bfJ_{ \mu, \shskip m } \hskip - 2 pt \lp x e^{i\phi} \rp  e (- 2 x y \cos (\phi + \theta) )    d x  d \phi
	= \frac {1} {4 y}   e\lp \frac {\cos \theta } y \rp  \bfJ_{  \frac 12 \mu, \shskip \frac 1 2 m } \lp  \frac 1 {
		16 y^2  e^{2 i \theta}  } \rp ,
	\end{split}
	\end{equation}
	for $y \in (0, \infty)$ and $\theta \in [0, 2 \pi)$, provided that $|\Re \mu| < \frac 1 2$ and $m$ is even. Here $ \bfJ_{ \mu, \shskip m }  (z)$ is the   Bessel function over the complex numbers defined as
	\begin{equation*}%\label{eq: defn of Bessel}
		\bfJ_{ \mu,\shskip   m} \lp z \rp =
		\left\{
		\begin{split}
			& \frac {2 \pi^2} {\sin (2\pi \mu)} \lp J_{\mu,\shskip   m} (4 \pi \sqrt z) -  J_{-\mu,\shskip   -m} (4 \pi \sqrt z) \rp, \hskip 6.5 pt \text {if } m \text{ is even},\\
			& \frac {2 \pi^2 i} {\cos (2\pi \mu)} \lp J_{\mu,\shskip   m} (4 \pi \sqrt z) + J_{-\mu,\shskip   -m} (4 \pi \sqrt z) \rp, \hskip 5 pt \text {if }  m \text{ is odd},
		\end{split}
		\right.
	\end{equation*}
with
\begin{equation*}%\label{0def: J mu m (z)}
	J_{\mu,\shskip   m} (z) = J_{- 2\mu - \frac 12 m } \lp  z \rp J_{- 2\mu + \frac 12 m  } \lp  {\overline z} \rp.
\end{equation*}
	See \S \ref{sec: defn of J(z)} for more discussions on the definition of  $ \bfJ_{ \mu, \shskip m }  (z)$.
	
	In this paper, we shall use the formula \eqref{eq: main 1} to establish the Bessel identity for the  Shimura-Waldspurger correspndence over $\BC$. This  completes the local spectral theory of  the relative trace formula of Jacquet,  complementary to \cite{BaruchMao-NA} and \cite{BaruchMao-Real}. As application of this paper,  we wish to further generalize the Waldspurger formula  onto an arbitrary number field\footnote{This was done very recently while the present paper was under peer review. See \cite{Chai-Qi-Wald}.}.
	
	\subsection{Main theorem}
	We now give a sample of the Bessel identities that we obtain.
	Let
	\begin{align*}
	N = \left\{ \begin{pmatrix}
	1 & z \\ & 1
	\end{pmatrix} : z \in \BC \right\}, \hskip 10 pt A =  \left\{ \begin{pmatrix}
	a &   \\ & b
	\end{pmatrix} : a, b \in \BC \smallsetminus \{0\} \right\}.
	\end{align*}
	Let $\psi_1 (z) = e (\Tr \, z )$, viewed  as a character on $N$. Let $\pi$ be an  infinite-dimensional  irreducible unitary representation of $\GL_2 (\BC)$  with trivial central character (that is, a representation of $  \PGL_2 (\BC)$). We attach to  $\pi$   the relative Bessel function  $ i_{\hskip 0.5 pt \pi, \hskip 0.5 pt \psi_1} $  on $\GL_2 (\BC)$ which  is left $A$-invariant  and right $(\psi_1, N)$-equivariant.
	 $ i_{\hskip 0.5 pt \pi, \hskip 0.5 pt \psi_1} $ is real analytic on an open subset of the Bruhat cell in $\GL_2 (\BC)$. Let $\sigma$ be
	an irreducible unitary representation of $\SL_2 (\BC)$.
	We attach to $\sigma$ the Bessel function $j_{\sigma, \hskip 0.5 pt \psi_1}$    on $\SL_2 (\BC)$   which    is both left and right $(\psi_1, N)$-equivariant.  $j_{\pi, \hskip 0.5 pt \psi_1}$ is real analytic on the open Bruhat cell in $\SL_2 (\BC)$. We stress that $\pi$ or $\sigma$ is determined by $ i_{\hskip 0.5 pt \pi, \hskip 0.5 pt \psi_1} $ or $ j_{\sigma, \hskip 0.5 pt \psi_1}$ respectively. Our main theorem of this
	paper (see Theorem \ref{thm: main}) is as follows.

	\begin{thm}\label{thm: main, intro} Let   $\pi $ be as above. There exists $\sigma$ as above such that for any $z \in \BC \smallsetminus \{0\}$ we have
		\begin{align}\label{0eq: Bessel identity}
		i_{\hskip 0.5 pt \pi, \hskip 0.5 pt \psi_1} \begin{pmatrix}
		z / 4 & 1 \\ 1 &
		\end{pmatrix}
		= \frac {2 \epsilon (\pi, 1/2) \, \psi_1     \lp   2 / z \rp |z|} {L(\pi, 1/2)}  j_{\sigma, \hskip 0.5 pt \psi_1} \begin{pmatrix}
		& - z\- \\ z \\
		\end{pmatrix},
		\end{align}
		in which $L (\pi, 1/2)$ and  $\epsilon (\pi, 1/2)$ are the central values of the $L$-factor and the $\epsilon$-factor associated with $ \pi $.
	\end{thm}
	
	The correspondence in Theorem \ref{thm: main, intro} is given by
	\begin{align*}
	\pi_{\mu, \shskip m}  \longrightarrow \sigma_{\frac 1 2 \mu, \shskip \frac 1 2 m}, \hskip 10 pt \text{for } \mu \in i \shskip \BR, m \text{ even, or } \mu \in \big( 0, \tfrac 1 2\big) \hskip -1 pt , m = 0,
	\end{align*}
	reflecting the index correspondence $(\mu, m) \longrightarrow \big(\tfrac 1 2 \mu, \tfrac 1 2 m \big)$ between the   Bessel functions on the two sides of \eqref{eq: main 1}. Here $\pi_{\mu, \shskip m}  $ and $\sigma_{\mu, \shskip m} $ are the unitary principal series or complementary series of $\GL_2 (\BC)$ and $\SL_2 (\BC)$ parameterized by $(\mu, m)$ respectively (see \S \ref{sec: representations of G and S} for the definitions). To be precise, we have
	\begin{align*}
	j_{\pi, \shskip \psi_1} \begin{pmatrix}
	&  z \\ 1& \end{pmatrix} =  
	& |  z| \bfJ_{ \mu,\shskip   m} \lp - z \rp, \hskip 10 pt \pi = \pi_{\mu, \shskip m}, 
	\end{align*}
	\begin{align*}%\label{4eq: j = J, SL}
	j_{\sigma, \shskip \psi_1} \begin{pmatrix}
	& - z\- \\ z \\
	\end{pmatrix}  =  (-1)^{\frac 1 2 m} \left|   z \right|^{-2}  \bfJ_{\frac 1 2  \mu,\shskip \frac 1 2   m} \lp  z^{-2} \rp, \hskip 10 pt \sigma = \sigma_{\frac 1 2 \mu, \shskip \frac 1 2 m},
	\end{align*} \footnote{It is preferred here to view $	j_{\pi, \shskip \psi_1}$ as function on $\PGL_2 (\BC)$($= \PSL_2 (\BC)$).}
	and we shall prove in the sense of distributions that 
	\begin{align*}
	i_{\hskip 0.5 pt \pi, \hskip 0.5 pt \psi_1} \begin{pmatrix}
	u & 1 \\ 1 &
	\end{pmatrix} = \frac 1 {L (\pi, 1/2)} \int_{ \BC \smallsetminus \{0\}} j_{\pi, \shskip \psi_1} \begin{pmatrix}
	& z \\ 1& \end{pmatrix} \psi_1 ( u z) \frac {d_1 z} {|z|^2 },
	\end{align*} 
	where $d_1 z$ denotes twice of the Lebesgue measure on $\BC$. Moreover, note that $ \epsilon (\pi, 1/2) = i^{|m|} = (-1)^{\frac 1 2 m} $ if $ \pi = \pi_{\mu, \shskip m} $. Thus the identity \eqref{0eq: Bessel identity} follows from the integral formula \eqref{eq: main 1}.
	
	Actually, we shall prove a more general identity between $i_{\hskip 0.5 pt \pi, \hskip 0.5 pt \psi}$ and $j_{\sigma, \hskip 0.5 pt \psi '}$ for any two nontrivial characters $\psi$ and $\psi'$.
The correspondence $\pi \ra \sigma  $ turns out  to be exactly the
Waldspurger correspondence $\pi \ra \Theta (\pi)$. However, unlike the real case as in \cite{BaruchMao-Real}, the correspondence is now independent on $\psi'$.

	\subsection{Remarks}
	
	Admittedly, the Waldspurger correspondence over $\BC$ is tremendously simpler than that over $\BR$, because all the double covers of $\SL_2 (\BC)$ are isomorphic to the trivial product $\SL_2 (\BC) \times \{\pm 1 \}$. Moreover, the representation theory of $\GL_2 (\BC)$ or $\SL_2 (\BC)$ is simpler as discrete series do not exist.
	
	Our expositions would be simplified if we view $\PGL_2 (\BC)$ as $\PSL_2 (\BC)$ and work only on $\SL_2 (\BC)$. Nevertheless, we choose to work in the framework of the representation theory of $\GL_2 (\BC)$ in order to preserve the analogy  between this work and \cite{BaruchMao-Real}.
	
	As the formula \eqref{eq: main 1} is the foundation of this paper, we now make some remarks on its analytic perspective and applications.
	
	The proof of \eqref{eq: main 1} in \cite{Qi-II-G} is considerably harder than that of \eqref{0eq: Weber's formula} or \eqref{0eq: Hardy's formula}. Interestingly, besides an incorporation of stationary phase and differential equations, also arise in the course of proof certain complicated combinatorial formulae.
	%We remark that the formula of the Bessel function $\bfJ_{ \mu, \shskip m }$ is deduced from the harmonic analysis by gamma factors and the Mellin transforms arising from the local functional equations for $\GL_2(\BC) \times \GL_1(\BC)$. More details may be found in \cite[\S 1-3]{Qi-Bessel}.
	
	It is the distributional variant of \eqref{eq: main 1} (see \eqref{6eq: Fourier}) that we shall use to deduce the formula of the relative Bessel function $i_{\pi, \shskip \psi}$. Critical is that the test function in \eqref{6eq: Fourier} only needs to be rapidly decaying at infinity but not zero.  By using this variant, we may
completely avoid the analysis of  differential operators as in \cite{BaruchMao-Real}. This idea would also work in the real context as in \cite{BaruchMao-Real}.
	
	\begin{acknowledgement}
		We thank Jim Cogdell for his help and the referee for several constructive suggestions.
	\end{acknowledgement}

	%\bigskip
%\footnotesize \noindent\textit{Acknowledgments.}

%\section*{Acknowledgments}
%The first author is supported by the National Natural Science
%Foundation of China [Grant 11771131].

%	\begin{rem}
%		By analyzing the asymptotics of certain orbital integrals as in \cite[\S 4, 5]{BaruchMao-Real}, it may be proven that Bessel and relative Bessel functions are locally integrable on the full group $G$.
%	\end{rem}
	
	\section{Notations and preliminaries}\label{sec: notations}
	
	\subsection{Basic notations}
	
	Let $G=\GL_2(\BC)$ and $S=\SL_2(\BC)$. Let $B $, $
	A  $ and $Z$  denote the Borel subgroup of upper triangular matrices,
	the diagonal subgroup and the center of $G$ respectively. Set
	\[
	N=\left\{n(x)=\begin{pmatrix}1&x \\ &1 \end{pmatrix}:x\in \BC
	\right\}, %\widebar N=\left\{\widebar n(y)=\begin{pmatrix}1& \\ y&1
	%\end{pmatrix}:y\in \BC \right\},
	\]
	and
	\begin{align*}
	  \varw    =  \begin{pmatrix} &-1
	\\ 1&
	\end{pmatrix} &, \hskip 10 pt \varw_0=\begin{pmatrix}
	&1 \\ 1& \end{pmatrix}, \\
	 s(a)   =\begin{pmatrix} a & \\ &a^{-1} \end{pmatrix} , \hskip 10 pt
	t(a)= &\begin{pmatrix} a& \\ &1 \end{pmatrix}, \hskip 10 pt z(c)=\begin{pmatrix} c& \\ &c
	\end{pmatrix}.
	\end{align*}
	
	Recall that $e (x) = e^{2 \pi i x}$. For $\lambdaup \in \BC$, let  $\psi = \psi_{\lambdaup}$ be   the additive
	character of $\BC$ defined by $$ \psi_{\lambdaup} (z)=e (\Tr \shskip (\lambdaup z)) = e (\lambdaup z + \overline {\lambdaup z}).$$ %, where $\widebar z$ is the complex conjugate of $z\in\BC$.
	 We also view $\psi$ as a character of $N$ by $\psi(n(z))=\psi(z)$.  Let $\|z\| = |z|_{\BC}= |z|^2$. Take  $dz = d_{\lambdaup} z$  to be $2 \sqrt { \|\lambdaup\|}$ times of the Lebesgue measure on $\BC$, which is self-dual with
	respect to $\psi$. Let $\BCx = \BC \smallsetminus \{0\}$. 
	Set $\dx  z=dz/\|z\| $.

   For $f\in L^1(\BC )$, we define the $\psi$-Fourier transform of $f$ as
   \begin{equation}\label{def: Fourier}
   \widehat f(u)= \int_\BC f(z)\psi(uz)dz.
   \end{equation}
   With our choice of  measure $d z$, the Fourier transform is self-dual, namely,  $\widehat  {\widehat f } (u) = f (-u)$.

  % In particular,
  % \[
 %  f(0)=\int_\BC \hat f(y)dy.
  % \]	

\subsection{Representations of $\GL_2 (\BC)$ and $\SL_2 (\BC)$}\label{sec: representations of G and S}

 %  For $z \in \BCx$, let $[z] = z/|z|$. 
 Denote by $\chiup_{\nu,\shskip l} $   the character of $\BCx$ given by
   \begin{align*}
    \chiup_{\nu,\shskip  l} : z \ra \|z\|^{  \nu} (z/|z|)^{l},
   \end{align*}
   with $\nu \in \BC$ and $ l \in \BZ $.  According to Langlands' classification for $G$, any irreducible admissible representation of $G  $ may be parametrized by $(\nu_1, \nu_2, l_1, l_2) \in \BC^2 \times \BZ^2$. First, we introduce the principal series representation $\pi (\chiup_{\nu_1, \shskip l_1}, \chiup_{\nu_2,\shskip l_2})$. It is known that $\mathrm{Ind}_B^G (\chiup_{\nu_1, \shskip l_1}, \chiup_{\nu_2,\shskip l_2}) $, defined by the unitary induction, has a unique irreducible quotient provided that  $\Re \nu_1 \geqslant \Re \nu_2$. We denote this quotient by $\pi (\chiup_{\nu_1, \shskip l_1}, \chiup_{\nu_2,\shskip l_2})$ according to \cite[\S 6]{J-L}. Second, these principal series $\pi (\chiup_{\nu_1, \shskip l_1}, \chiup_{\nu_2,\shskip l_2})$ exhaust all the irreducible admissible representations of $G$ up to infinitesimal equivalence that occurs when we permute $ \chiup_{\nu_1, \shskip l_1}$ and $ \chiup_{\nu_2,\shskip  l_2}$. Note that we can always permute $( \nu_1, l_1 )$ and $(\nu_2, l_2)$ if necessary so that $\Re \nu_1 \geqslant \Re \nu_2$ is satisfied.  See \cite[\S 4]{Knapp} and  \cite[\S 6]{J-L} for more details.

For $\mu\in \BC, m\in \BZ$, we let $ \pi_{\mu,\shskip m}$ denote the principal series   of $G$ with parameter $\big(\mu, - \mu, \frac 1 2 m, - \frac 1 2 m \big)$ if $m$ is even or $ \big(\mu, - \mu, \frac 1 2 (m+1), - \frac 1 2 (m-1) \big) $ if $m$ is odd. Note that $\pi_{\mu, \shskip  m} \otimes \chiup_{\nu,\shskip   l} (\det)$ exhaust all the principal series of $G$   as above. It is clear that $\pi_{\mu, \shskip  m}$ has trivial central character if and only if $m$ is even. %The central character of $\pi_{\mu, \shskip  m}$ is trivial if $m$ is even and $[\cdot]$ if $m$ is odd.

When restricting $\pi_{\mu, \shskip  m}$ on $S$, we obtain the principal series $\sigma_{\mu,\shskip m}$ of $S$ induced from the character $s(a) \ra \chiup_{ 2 \mu,\shskip  m} (a) $. The principal series $\sigma_{\mu,\shskip m}$ exhaust all  the irreducible admissible representations of $S$.

Finally, the representation $\pi_{\mu, \shskip  m}$ or $\sigma_{\mu, \shskip  m}$ is unitary if
\begin{itemize}
	\item[-] (unitary principal series) $\Re \mu = 0$, or
	\item[-] (complementary series) $\mu   \in %\lp  - \frac 1 2, 0 \rp \cup
	\lp 0, \frac 1 2\rp$ and $ m = 0$.
\end{itemize}
For more information, the reader may consult the book of Knapp \cite{Knapp-Book}.
%Note that complementary series arise from   non-unitary principal series.

\subsection{Whittaker functions}

	Let $\pi$ be an infinite dimensional irreducible unitary
representation of $G$  on a Hilbert space $H$, with $H_\infty$ its subspace of smooth vectors. It is well known that $\pi$ is generic in the sense that
there exists a  nonzero continuous $\psi$-Whittaker
functional $L$ on $H_\infty$, unique up to scalars, satisfying
\begin{equation*}
L (\pi(n )\varv)=\psi(n )L(\varv), \hskip 10 pt n \in N, \   \varv\in H_\infty.
\end{equation*}
Let
\begin{equation}\label{def: Wv(g)}
W_\varv(g)=L(\pi(g)\varv), \hskip 10 pt  \varv\in H_\infty, \  g\in G,
\end{equation}
be the Whittaker function corresponding to $\varv$. %The space of Whittaker functions
%\[
%\CW(\pi,\psi)=\left\{W_\varv :\varv\in H_\infty \right\}
%\]
%is called the Whittaker model of $\pi$.  
All these definitions are valid for  $S$.

First, we have the following lemma for the asymptotic of Whittaker functions on the torus. 

\begin{lem}\label{lem: asymptotic of W(t(a))}
	For  $\varv\in H_\infty$, the function  $W_{\varv} (t(a))$ is rapidly decreasing at infinity and is of order $\|a\|^{\rho}$  for certain $\rho > 0$ when $a$ is in the vicinity of zero. 
\end{lem}

This lemma is a consequence of a much more general result of Jacquet and Shalika for $\GL_n $ over a local field in \cite[\S 4.4 Proposition 3]{Jacquet-Shalika:Exterior}. Specialized to the case of $\GL_2 (\BC)$, their result may be phrased as follows. There is a finite set $C$ of characters of $\BCx$ with {\it positive} real part, namely, characters $\chiup $ with $|\chiup (z)| = \|z\|^{\rho}$ for $\rho > 0$, and for each $\chiup \in C$ a nonnegative integer $r_{\chiup}$ with the following property\footnote{For $\GL_2$, one should have $r_{\chiup} = 0$ or $1$.}: let $X$ be the set of finite functions of the form $\chiup (a) ( \log \|a\|)^{r}$, with $\chiup \in C$ and $r \leqslant r_{\chiup}$, then for   any given $\varv \in H_{\infty}$ there are functions $\phi_{\xi}$ in the Schwartz space on $\BC \times \mathrm{U}_2(\BC)$ such that 
\begin{align*}
W_{\varv} (t(a) k) = \sum_{ \xi\shskip \in X } \phi_{\xi} (a, k) \xi (a).
\end{align*} 
The rapid decay of $W_{\varv} (t(a))$ as $\|a\| \ra \infty$ is already proven in   \cite[\S 2.5]{Godement} (and also \cite[\S 6]{J-L}), in particular (68)-(72). However, we do not find any concrete statement in either \cite{J-L} or \cite{Godement} on the asymptotic of $W_{\varv} (t(a) )$ for $a$ small. Nevertheless, we may prove the asymptotic $W_{\varv} (t(a)) = O (\|a\|^{\rho}) $ as $\|a\| \ra 0$ if we examine and estimate the integrals in (69) and (70) in   \cite[\S 2.5]{Godement} more carefully; it is important here that $\pi$ is unitary so that either $\Re \mu = 0$ or $\mu \in \lp 0, \frac 1 2 \rp$ (we may choose $0 < \rho < \frac 1 2 - \Re \mu$). It should be noted that only $\mathrm{U}_2 (\BC)$-finite vectors are treated in \cite[\S 6]{J-L} and \cite[\S 2.5]{Godement}.

Moreover, we have the following analogue of    \cite[Lemma 2.1]{BaruchMao-Real}.
\begin{lem}\label{le21}
Let $\varv\in H_\infty$, $f,g\in C^\infty(\BCx)$. Assume that both $f(a)W_\varv(t(a))$ and $g(a)W_\varv(t(a))$ are in $L^1(\BCx, \dx  a)$. If
\[
\int_{\BCx} f(a)W_{\pi(n)\varv}(t(a))\dx  a=\int_{\BCx} g(a)W_{\pi(n)\varv}(t(a))\dx  a
\]
for all $n\in N$, then $f(a)W_\varv(t(a))=g(a)W_\varv(t(a))$ for all $a\in \BCx$.
\end{lem}

\begin{proof}
	The proof is literally the same as that of \cite[Lemma 2.1]{BaruchMao-Real}. We have
	\begin{align*}
	\int_{\BCx} f(a)W_{\pi(n(x))\varv}(t(a))\dx  a  = \hskip -1 pt \int_{\BCx} f(a)W_{ \varv}(t(a) n(x))\dx  a  = \hskip -1 pt \int_{\BCx} f(a)W_{ \varv}(t(a) ) \psi (ax) \dx  a.
	\end{align*}
	Hence follows the integrability of the first integral for all $x \in \BC$ and it is equal to the Fourier transform of the function $\|a\|\- f(a)W_{ \varv}(t(a) ) $. The identity in the lemma then yields the equality between the Fourier transform of $\|a\|\- f(a)W_{ \varv}(t(a) ) $ and that of $\|a\|\- g (a)W_{ \varv}(t(a) ) $. It follows that $f(a)W_\varv(t(a))=g(a)W_\varv(t(a))$ for all $a\in \BCx$.
\end{proof}

	\section{Bessel and relative Bessel distributions for $\GL_2 (\BC)$}\label{sec: distributions}

	%Now assume that $\pi$ has trivial central character.
	Let $\langle \,  ,\shskip \rangle $ be a $G$-invariant nonzero inner product on $H$. We can
	normalize the Whittaker functional $L$ so that
	\begin{equation}\label{eq: inner product}
	\langle \varv_1,\varv_2\rangle =\int_{\BCx}W_{\varv_1}(t(a))\overline
	{W_{\varv_2}(t(a))}\dx a.
	\end{equation}
	
	\subsection{The normalized torus invariant functional}
	 We define
	\begin{equation}\label{def: P}
	P(\varv)=\frac{1}{L(\pi,1/2)}\int_{\BCx} W_\varv(t(a))\dx a, \hskip 10 pt \varv \in H_{\infty},
	\end{equation}
	in which the normalization factor $L(\pi, 1/2)$ is the central value of the $L$-function for $\pi$ (see for example \cite[\S 4]{Knapp}).
In view of Lemma \ref{lem: asymptotic of W(t(a))}, the integral above is absolutely convergent. Moreover, it may be verified that $P$ is a  nonzero continuous linear
	functional on $H_\infty$ satisfying
	\[
	P(\pi(a)\varv)= P(\varv), \hskip 10 pt a \in A, \ \varv\in H_\infty,
	\]
	 if the central character of $\pi $ is trivial.
	Hence $P$ is the nonzero   unique up to scalar continuous linear
	functional with the above invariance property.

	\subsection{Bessel and relative Bessel distributions}
	For every continuous linear functional $\lambda $ on
	$H_\infty$ and every $f\in C_c^\infty(G)$, the map $\varv\to \lambda(\pi(f)\varv)$ is continuous on $H$ by  \cite[Proposition 3.2]{Shalika-MO}. By the Riesz representation theorem, there exists a unique
	vector $\varv_{\lambda,\shskip f}$ such that
	\begin{equation}\label{def: vector v l f}
	\lambda (\pi(f)\varv)=\langle \varv,\varv_{\lambda,\shskip f}\rangle , \hskip 10 pt \text{ all }  \varv\in H_\infty.
	\end{equation}
	More concretely, for any orthonormal basis $\{ \varv_{\mathfrak i} \}$ in $H_{\infty}$, 
	\begin{align}\label{3eq: orthonormal basis}
\varv_{\lambda,\shskip f} = \sum  \overline {\lambda (\pi(f)\varv_{\mathfrak i})} \shskip \varv_{\mathfrak i} .
	\end{align}
	Then we  define the Bessel distribution by
	\begin{equation}\label{def: J (f)}
	J(f)=J_{\pi,\shskip \psi}(f)=\overline{L(\varv_{L, \shskip f})}
	\end{equation}
	and the relative Bessel distribution by
	\begin{equation}\label{def: I(f)}
	I(f)=I_{\pi,\shskip \psi}(f)=\overline{L(\varv_{P, \shskip f}) }.
	\end{equation}
	In view of \eqref{3eq: orthonormal basis}, for any orthonormal basis $\{ \varv_{\mathfrak i} \}$ in $H_{\infty}$, we have
	\begin{align}\label{3eq: J(f)}
	J (f) = \sum  L (\pi(f)\varv_{\mathfrak i}) \overline {L (\varv_{\mathfrak i})}, 
	\end{align}
	and
	\begin{align}\label{3eq: I(f)}
	I (f) = \sum  P (\pi(f)\varv_{\mathfrak i}) \overline {L (\varv_{\mathfrak i})}.
	\end{align}
	Letting $\widetilde f(g)=f(g^{-1})$, it follows from \cite[Corollary 23.7]{BaruchMao-Real} that $J(f)=L(\varv_{L,\shskip\widetilde f})$ and $I(f)=P(\varv_{L,\shskip\widetilde f})$.% We use $\rho_l$ and $\rho_r$ to denote the left
	%and right translations of functions on $G$.
	Thus
	\[
	I(f)=P(\varv_{L,\shskip\widetilde f})=\frac{1}{L(\pi,1/2)}\int_{\BCx}
	W_{\varv_{L,\shskip \widetilde f}}(t(a))\dx  a=\frac{1}{L(\pi,1/2)}\int_{\BCx}
	L(\pi(t(a))\varv_{L,\shskip \widetilde f})\dx a.
	\]
	Let $\rho_l$ denote the left translation of functions on $G$, namely, $(\rho_l (h) f)(g) = f (h\- g)$. 
	Note that for any $\varv\in H_\infty$
	\begin{align*}
	\pi((\rho_l(t(a))f)^{\sim}) \varv 
	& =    \int_G f(t(a\-) g\-)\pi(g) \varv \shskip dg \\
	& =\int_G f(g\-)\pi(g)\pi(t(a\-)) \varv\shskip dg = \pi (\widetilde f) \pi (t(a\-)) \varv,
	\end{align*}
	and therefore $\pi(t(a))\varv_{L,\shskip\widetilde f}=\varv_{L,\shskip (\rho_l(t(a))f)^{\sim}}$. Hence
	\begin{equation}\label{eq: I(f) = int J}
	I(f)=\frac{1}{L(\pi,1/2)}\int_{\BCx}J(\rho_l(t(a))f)\dx a.
	\end{equation}
	
    Finally, we define   Bessel distributions for $S$ in the same manner. Let $(\sigma, H)$ be a   unitary representation of $S$. Let $L$ be a nonzero continuous $\psi$-Whittaker functional on $H_\infty$. Similar as above, for any $h\in C_c^\infty(S)$, there exists a unique vector $\varv_{L,\shskip  h} \in H$ such that
   \begin{equation}\label{def: v l phi, S}
    L (\varv_{L,\shskip h})=\langle \varv,\varv_{L,\shskip  h} \rangle, \hskip 10 pt \text{ all } \varv \in H_\infty.
   \end{equation}
    Then one can define similarly the Bessel distribution $J_{\sigma,\shskip \psi}$ on $S$ associated with $\sigma$ by
    \begin{equation}\label{def: J(phi), S}
    J_{\sigma,\shskip \psi}(h)=\overline {L (\varv_{L,\shskip h})}.
    \end{equation}

	\section{Bessel functions for $\GL_2 (\BC)$}\label{sec: def of Bessel j}

	Bessel functions for $\PSL_2 (\BC)$ were first discovered by  Bruggeman and Motohashi \cite{B-Mo}, and later by Lokvenec-Guleska \cite{B-Mo2} for $\SL_2 (\BC)$, arising in their  Kuznetsov trace formulae (Bessel functions for spherical representations of $\SL_2 (\BC)$ however appeared much earlier in the work of Miatello and Wallach \cite{M-W-Kuz}).  Recently, the second author rediscovered Bessel functions for $\GL_2 (\BC)$ as an example of the Bessel functions for $\GL_n (\BC)$ occurring in the \Voronoi summation formula; see \cite[\S 3, 15, 17, 18]{Qi-Bessel}. 
	
	In the representation theoretic aspect, most important is a kernel formula in \cite[\S 18]{Qi-Bessel} for the action of the Weyl element in the Kirillov or Whittaker model of an irreducible unitary representation of $\GL_2 (\BC)$. This action is given by the Hankel transform over $\BCx$ with integral kernel the associated Bessel function. Such a kernel formula for $\GL_2 (\BR)$ and $\GL_2 (\BC)$ lies in the center of the representation theoretic approach to the Kuznetsov trace formula; see \cite{CPS} and \cite{Qi-Kuz}. In the case of $\GL_2 (\BR)$ or $\SL_2 (\BR)$, there are three proofs of the kernel formula in \cite[\S 8]{CPS},  \cite{Mo-Kernel} and \cite[Appendix 2]{BaruchMao-Real}. Methods in the latter two proofs were generalized onto  $\SL_2 (\BC)$ in  \cite{B-Mo-Kernel2,Mo-Kernel2} and \cite{Baruch-Kernel}, but there are  unpleasant restrictions on both of them  due to some  convergence issues. In \cite{B-Mo-Kernel2}, an integral representation of the Bessel function is used but it is valid only for  $ |\Re \mu | < \frac 1 8$. In \cite{Baruch-Kernel}, it requires that $\Re \mu  \neq 0$ and that the functions in the Kirillov model are compactly supported.\footnote{It should be noted that our parametrization is slightly different from theirs.} The approach in \cite{Qi-Bessel} is quite different. It is based on  the  sophisticated harmonic analysis by gamma factors and the Mellin transforms (see \cite[\S 1-3]{Qi-Bessel}). Also the ideas in  \cite[\S 8]{CPS} are followed and generalized  in \cite[\S 17]{Qi-Bessel} to $\GL_n (\BR)$ and $\GL_n (\BC)$.
	
	\subsection{The definition of $  \bfJ_{\mu,\shskip   m} (z)$} \label{sec: defn of J(z)}
	Let $\mu \in \BC $   and $m \in \BZ $.
	We define
	\begin{equation}\label{0def: J mu m (z)}
	J_{\mu,\shskip   m} (z) = J_{- 2\mu - \frac 12 m } \lp  z \rp J_{- 2\mu + \frac 12 m  } \lp  {\overline z} \rp.
	\end{equation}
	%with $J_{\nu} (z)$ the Bessel function of the first kind of order $\nu$.
	The function $J_{\mu,\shskip   m} (z)$ is well defined in the sense that the   expression on the right of \eqref{0def: J mu m (z)} is independent on the choice of the argument of $z$ modulo $2 \pi$. Next, we define
	\begin{equation}\label{0eq: defn of Bessel}
	\bfJ_{ \mu,\shskip   m} \lp z \rp =
	\left\{
	\begin{split}
	& \frac {2 \pi^2} {\sin (2\pi \mu)} \lp J_{\mu,\shskip   m} (4 \pi \sqrt z) -  J_{-\mu,\shskip   -m} (4 \pi \sqrt z) \rp, \hskip 6.5 pt \text {if } m \text{ is even},\\
	& \frac {2 \pi^2 i} {\cos (2\pi \mu)} \lp J_{\mu,\shskip   m} (4 \pi \sqrt z) + J_{-\mu,\shskip   -m} (4 \pi \sqrt z) \rp, \hskip 5 pt \text {if }  m \text{ is odd},
	\end{split}
	\right.
	\end{equation}
	where $\sqrt z$ is the principal branch of the square root, and it is understood that in the nongeneric case when  $4 \mu \in 2\BZ + m$ the right hand side should be replaced by its limit.  $\bfJ_{ \mu,\shskip   m} \lp z \rp$ is a well defined function on $\BCx$ only when $m$ is even, but it becomes    well defined after multiplying the factor $\sqrt {  z / |z|}  $ when $m$ is odd.
	
	For later use, the following crude estimates for $\bfJ_{ \mu,\shskip   m} \lp z \rp$  are sufficient. See for example  \cite[(2.28)]{Qi-II-G}. %\footnote{Note that (2.23) in \cite{Qi-II-G} holds only for generic $\mu $ but one only loses a $\log |z|$ for non-generic $\mu$. The conscientious readers are referred to  \cite[\S 5.2]{Qi-Wilton} for a proof by estimating   Mellin-Barnes type integrals of   gamma factors.}. %In the unitary case, we   have either $\Re \mu = 0$ or $\mu \in \big(0, \tfrac 1 2 \big)$ so that the $\rho  $ below may always be chosen to be $\tfrac 1 2 $.
	
	\begin{lem}\label{4lem: bounds for J}
		Suppose that $|\Re \mu | < \rho$. We have
		\begin{equation*}
		\bfJ_{ \mu,\shskip   m} \lp z \rp \lll \left\{\begin{split}
		& 1 / \left| z  \right|^{2 \rho}, \hskip 11 pt \text{ if } |z| \leqslant 1, \\
		& 1 /{\textstyle \sqrt {|z|}}, \hskip 10 pt \text{ if } |z| > 1,
		\end{split}\right.
		\end{equation*}
		with the implied constant depending only on $\mu$, $\rho$ and $m$. In particular,
		\begin{equation*}
		\bfJ_{ \mu,\shskip   m} \lp z \rp \lll   1 / \left| z  \right|^{2 \rho} + 1 / {\textstyle \sqrt {|z|}}.
		\end{equation*}
	\end{lem}
	
	\subsection{Bessel functions for $\GL_2 (\BC)$}\label{sec: Bessel, GL2}
	As defined in \S \ref{sec: notations}, let $\pi = \pi_{\mu, \shskip  m}$ be a principal series representation of $G = \GL_2 (\BC)$, and, for $\lambdaup \in \BCx$, let $\psi = \psi_{\lambdaup} $ be a nontrivial additive character on $N$. The central character $\omega_{\pi} (z (c)) = c /|c|$ if $m$ is odd and $ \omega_{\pi} (z (c)) \equiv 1$ if $m$ is even. We define a function $j = j_{\pi, \shskip \psi}$ supported on the open Bruhat cell $X = B \varw_0 B$ such that
	\begin{equation}\label{4eq : j = J}
	j_{\pi, \shskip \psi} ( t(a) \varw_0 ) = \left\{
	\begin{split}
	& |\lambdaup a| \bfJ_{ \mu,\shskip   m} \lp - \lambdaup^2 a \rp,  \hskip 42 pt \text {if } m \text{ is even},\\
	&  - i |\lambdaup| {\textstyle \sqrt {    |a| a}} \bfJ_{ \mu,\shskip   m} \lp - \lambdaup^2 a \rp,  \hskip 5 pt \text {if }  m \text{ is odd},
	\end{split}
	\right.
	\end{equation}
	 \footnote{In \cite{Qi-Bessel}, as well as  \cite{CPS}, the measure  is not normalized by the factor $\sqrt {\|\lambdaup\|}$ and   the choice of   Weyl element  is $\varw$ instead of $\varw_0$, so the Bessel functions in \cite[Proposition 18.5]{Qi-Bessel} are slightly different.} and that $j_{\pi, \shskip \psi}$ is left and right $  (\psi, N)$-equivariant and also $(\omega_{\pi}, Z)$-equivariant, namely,
		\begin{align}\label{4eq: psi variant}
		j_{\pi, \shskip \psi} \lp n(x)z(c) t(a) \varw_0 n(y) \rp = \psi (x) \psi (y) \omega_{\pi} (c ) j_{\pi, \shskip \psi} ( t(a) \varw_0 ).
		\end{align}
		According to \cite[\S 18]{Qi-Bessel}, we have
	\begin{equation}\label{4eq: Weyl element action}
	W_{\varv} \lp
	t (b)
	\varw_0 \rp
	=  \int_{\BCx } \omega_\pi (  a  )\-  j ( t(ab) \varw_0) W_{\varv} \lp
	t (a)
	\rp \dx  a, \hskip 10 pt \text{ all }  \varv \in H_{\infty}.
	\end{equation}
As an easy consequence of \eqref{4eq: psi variant} and \eqref{4eq: Weyl element action}, we have the following formula.
	
	\begin{thm}\label{thm: W(g) on BwB}
		Let $\varv \in H_{\infty}$. Then
		\begin{align}
		W_{\varv} \lp
		g \rp
		=  \int_{\BCx } j \big( g t \big( a\- \big) \big)  W_{\varv} \lp
		t (a)
		\rp \dx  a,
		\end{align}
		for all $ g  \in B \varw_0 B$.
	\end{thm}

	%\subsection{Bessel Functions for $\SL_2 (\BC)$}
	  It readily follows from (\ref{4eq : j = J}, \ref{4eq: psi variant}) that the restriction of the Bessel function $j_{\pi, \shskip \psi}$ on   $S = \SL_2 (\BC)$  is given by
	\begin{align}\label{4eq: j = J, SL}
	j_{\pi, \shskip \psi} (  s (a) \varw ) =  (-1)^m \left|\lambdaup   a^2 \right|  \bfJ_{ \mu,\shskip   m} \lp   \lambdaup^2   a^2 \rp,
	\end{align}
	and
	\begin{align}\label{4eq: psi variant, SL}
	j_{\pi, \shskip \psi} \lp n(x)   s (a) \varw n(y) \rp = \psi (x) \psi (y)  j_{\pi, \shskip \psi} (  s (a) \varw ).
	\end{align}
	
	\begin{rem}
		Of course, the factor $ \omega_\pi (  a  )\-$ in \eqref{4eq: Weyl element action} disappears if $\pi $ is a representation of $\PGL_2 (\BC)$ in the case when $m$ is even. Otherwise, it will be gone if one uses $s (a)$ instead of $t (a)$ for torus elements, so the kernel formula  looks simpler for $\SL_2 (\BC)$,
		\begin{equation}%\label{4eq: Weyl element action}
		W_{\varv} \lp
		s (b)
		\varw  \rp
		=  \int_{\BCx }   j  ( s(ab) \varw ) W_{\varv} \lp
		s (a)
		\rp \dx  a,
		\end{equation}
		for all $\varv \in H_{\infty}$.
		See \cite[Theorem 2]{Mo-Kernel} and \cite[Theorem 2.3]{Baruch-Kernel}.
	\end{rem}
    We extend the definition of $j $ to the whole group $G$ by setting $j (g)=0$ for $g\notin X $.
\begin{prop}\label{thm: local integrability}
		Assume that $\pi $ is unitary. Then $j $ is locally integrable over $G$. Namely, for any $f\in C_c^\infty(G)$, we have
\[
\int_G \left|f(g)j (g) \right| dg<\infty.
\]
	\end{prop}
\begin{proof} Recall that the measure on $X = NA \varw_0 N$ is given by $ d g = |a|^{-2} d x \hskip 1 pt \dx  a   \hskip 1 pt \dx c \hskip 1 pt d y$ if $g = n(x)z(c)t(a)\varw_0n(y)$. We have
\begin{align*}
 \int_G |f(g)j (g)| dg
= \ & \int_{X}|f(g)j (g)|dg \\
= \ &\int |j (t(a)\varw_0)|\left(\iiint |f(n(x)z(c)t(a)\varw_0n(y))|dxdy\dx  c \right) |a|^{-2}\dx  a \\
= \ & \int |j (t(a)\varw_0)| \lp  \int  J \big( \sqrt a,   c, |f| \big)    \dx  c  \rp |a|^{-2} \dx  a,
\end{align*}
where $ J \big( \sqrt a,   c, |f| \big) $ is the orbital integral defined as in \eqref{def: J (a,c,f)}.  In view of Proposition \ref{prop:A1},  the inner integral over $c$ is zero for small $|a|$    and of the order $|a|^{1+\varepsilon}$ for large  $|a|$. In view of  Lemma \ref{4lem: bounds for J} and the expression of $j = j_{\pi, \shskip \psi}$ as in \eqref{4eq : j = J}, $j   (t(a) \varw_0)$ is bounded by $ {|a|}^{\frac 1 2}$ when $|a|$ is large. It is now clear that the integral above is convergent.
\end{proof}

	\section {Bessel distributions for $\GL_2(\BC)$}
	
	In this section, we show that the Bessel distribution
	$J=J_{\pi,\shskip \psi}$ %attached to an irreducible unitary representation	$\pi$
	  is represented by the Bessel function $j=j_{\pi,\shskip \psi}$. %We remark that most proofs are similar to those in \cite[\S 7]{BaruchMao-Real}\footnote {Lemma 7.2 in \cite{BaruchMao-Real} is weaker than  our Lemma \ref{le42} due to an error in their proof.}, but  the details will be presented here for the sake of completeness.
	
	\begin{lem}{\label{le41}}
		Let $\varv\in H_\infty$, then
		\[
		\int_{\BCx}J(\rho_r(t(a))f)W_{\varv}(t(a))\dx a=\int_Gf(g)W_{\varv}(g)dg,
		\]
		in which $\rho_r$ is the right translation, that is, $(\rho_r (h) f)(g) = f (  g h)$.
	\end{lem}
   \begin{proof}The proof is similar to that of   \cite[Lemma 7.1]{BaruchMao-Real}.
   For any $\varv\in H_\infty$, we have
   \begin{align*}
    \pi(\rho_r(t(a))f) \varv % & =\int_G (\rho_r(t(a)) f)(g)\pi(g) \varv \shskip dg   \\
    & =    \int_G f(gt(a))\pi(g) \varv \shskip dg \\
    & =\int_G f(g)\pi(g)\pi(t(a^{-1})) \varv\shskip dg = \pi (f) \pi(t(a^{-1})) \varv.
   \end{align*}
   Thus it follow from  \eqref{def: vector v l f}, \eqref{def: J (f)} and \eqref{def: Wv(g)} that $\varv_{L,\hskip 1 pt \rho_r(t(a))f}=\pi(t(a))\varv_{L, \shskip f}$ and $J (\rho_r(t(a))f)=\overline {W_{\varv_{L, \shskip f}}(t(a))}$. Hence, in view of \eqref{eq: inner product}, we have
   \[
   \int_{\BCx} J (\rho_r(t(a))f)W_\varv(t(a))\dx  a=\langle \varv,\varv_{L, \shskip f} \rangle.
   \]
   On the other hand, by \eqref{def: vector v l f} and \eqref{def: Wv(g)},
   \[
   \langle \varv,\varv_{L, \shskip f} \rangle =L(\pi(f)\varv)=\int_G f(g)W_\varv(g)dg,
   \]
   which finishes the proof.
   \end{proof}
   Recall the definition of the Bessel function $ j = j_{\pi,\shskip \psi} $ in the last section. We define a distribution $\widetilde J = \widetilde J_{\pi,\shskip \psi}$ on $C_c^\infty(G)$ by
   \[
   \widetilde J (f)=\int_{G}f(g)j(g)dg.
   \]	
   This distribution is well defined as we have proven in Proposition \ref{thm: local integrability} that $j_{\pi,\shskip \psi}(g)$ is locally integrable on $G$.
\begin{lem}
{\label{le42}}\footnote{Lemma \ref{le42} is stronger than Lemma 7.2 in \cite{BaruchMao-Real}  as our computation shows that it is not necessary to assume that $W_{\varv} (t(a))$ has  a high order of vanishing at $a = 0$.} Let $\varv\in H_\infty$ and $f\in C_c^\infty(G)$. %Assume that $W_{\varv}(t(a))=O(|a|^2)$ at $a=0$.
Then
\[
\int_{\BCx} \widetilde J (\rho_r(t(a)) f)W_{\varv}(t(a))\dx  a=\int_G f(g)W_{\varv}(g)dg.
\]
\end{lem}
\begin{proof} We have
\begin{align*}
& \hskip 13 pt  \int  \widetilde J (\rho_r(t(a)) f)W_{\varv}(t(a))\dx  a \\
&= \int \left(\int_{X}f(gt(a))j (g)dg\right) W_\varv(t(a)) \dx  a \\
&= \int \left( \int_{X}f(g)j (gt(a^{-1}))dg \right)  W_\varv(t(a))\dx  a  \\
&= \int_{X} f(g) \left( \int  j (gt(a^{-1}))W_\varv(t(a))\dx  a \right)dg  \\
&= \int_{X} f(g)W_\varv(g)dg.
\end{align*}
Here we have obtained the last equality from \eqref{thm: W(g) on BwB} in Theorem \ref{thm: W(g) on BwB}. It is however needed to justify the change of order of integrations in the second to the last equality. For this, it suffices to verify the absolute convergence of the integral in the third line. Note that
\begin{align*}
&\int_{X}|f(g)j (gt(a^{-1}))|dg \\
= \ &\int |j (t(ab)\varw_0)|\left(\iiint |f(n(x)z(c)t(b)\varw_0n(y))|dxdy\dx  c \right) |b|^{-2}\dx  b \\
= \ & \int |j (t(ab)\varw_0)| \lp  \int  J \big( \sqrt b,   c, |f| \big)    \dx  c  \rp |b|^{-2} \dx  b,
\end{align*}
where $ J \big( \sqrt b,   c, |f| \big) $ is the orbital integral defined as in \eqref{def: J (a,c,f)}.  In view of Proposition \ref{prop:A1}, the inner integral over $c$ is  dominated by  $  |b|^{1 + \varepsilon} \Upsilon_B (b)$ for some    constant $B$, with $\Upsilon_B$ the characteristic function on $\{ b: |b| \geqslant B \}$. Since $\pi = \pi_{\mu, \shskip  m}$ is unitary, we have   $|\Re \mu | <   \frac 1 2$. In view of  Lemma \ref{4lem: bounds for J} and the expression of $j = j_{\pi, \shskip \psi}$ as in \eqref{4eq : j = J}, $j   (t(a) \varw_0)$ is bounded by $1 +   {|a|}^{\frac 1 2}$. Consequently, we have the following estimations for the integral above
\begin{align*}
\lll \int   \big(  1 +   {|ab|}^{\frac 1 2} \big) |b|^{- 3 + \varepsilon} \Upsilon_B (b) d b \lll 1 + |a|^{\frac 1 2}.
\end{align*}
Finally, recall from Lemma \ref{lem: asymptotic of W(t(a))} that $W_{\varv} (t(a))$ is rapidly decreasing at infinity and of the order $|a|^{2\rho}$ near zero for certain $\rho > 0$, then follows the absolute convergence of the integral.
\end{proof}

    \begin{cor}{\label{cor53}}
		The two distributions $\widetilde J $ and $J $ are the same. That is, for any $f\in C_c^\infty(G)$, we have
\[
J (f)=\int_G f(g)j (g)dg.
\]
	\end{cor}
    \begin{proof}
    Choose $\varv\in H_\infty$ with $W_{\varv} (1) = 1$, say. %such that its Whittaker function $W_\varv(t(a))$ has a high order of vanishing at $a=0$ and $W(1)=1$.
    By Lemma \ref{le41} and   \ref{le42} we have
    \[
    \int  J (\rho_r(t(a))f)W_\varv(t(a))\dx  a=\int  \widetilde J (\rho_r(t(a))f)W_\varv(t(a))\dx  a
    \]
    for all $f\in C_c^\infty(G)$. We replace $\varv$ by $\pi(n)\varv$ for any  $n\in N$ and apply Lemma \ref{le21}, it follows that
    \[
    J (\rho_r(t(a))f)W_\varv(t(a))=\widetilde J (\rho_r(t(a))f)W_\varv(t(a))
    \]
    for all $a\in \BCx$. Letting $a = 1$, the conclusion follows immediately.
    \end{proof}

Let $\sigma=\sigma_{\mu, \shskip m}$ and $\pi=\pi_{\mu, \shskip m}$ be unitary principal series of $S$ and $G$ respectively; see \S \ref{sec: representations of G and S}. Recall that $\sigma $ is the restriction of $\pi$ on $S$. Let $H$   be their common underlying space. Let  $L$ be a fixed common $\psi$-Whittaker functional. Let  the  Bessel distribution $J_{\sigma,\shskip \psi}$  be defined by (\ref{def: v l phi, S}, \ref{def: J(phi), S}). We now prove that the  Bessel distribution $J_{\sigma,\shskip \psi}$  is represented by the restriction of the Bessel function $j_{\pi,\shskip \psi}$ to $S$.
	
    \begin{prop}\label{prop54: regularity of J on S}	
    Let $\sigma $ and $\pi$ be as above.  For any $h\in C_c^\infty(S)$, we have
    \[
    J_{\sigma,\shskip \psi}(h)=\int_S h(s) j_{\pi,\shskip \psi}(s) ds.
    \]
    As such, we shall write $j_{\sigma, \shskip \psi} $ the restriction of $j_{\pi,\shskip \psi}$ on $S$.
    \end{prop}
   \begin{proof} %Let notations be as in \S \ref{sec: distributions}.
   Fix $h\in C_c^\infty(S)$. Let  $U \subset Z$ be a small open neighborhood of the identity  $1$ so that the map $ (s, z) \ra s z $ is an injection from $S \times U$ into $G = S \cdot Z$. %Note that $G \cong S\times Z/\{\pm 1 \}  $ and that $Z /  \{\pm 1\} \cong \BCx$ via the   map $z(c) \ra c^2$. Choose
   Choose a function $\tw \in C_c^{\infty} (U)$ with
   \begin{align*}
  \int_{U} \tw(z) d  z = 1,
   \end{align*}
   where $d z = \dx  c$ if $z = z(c)$.
   %Set $f (s z) = h(s) \omega_{\pi}\- (z  ) \tw (z)$.  %Set $$2 f (s z(c) ) =  \lp h (s) \omega_{\pi}\- (z(c)) + h (- s) \omega_{\pi}\- (z(-c))  \rp \tw (c^2).$$
   Set $f (s z) = h (s) \omega_{\pi}\- (z) \tw (z) $.
%   \begin{equation*}
%		f(g)= \left\{\begin{split}
%		& h(s)\tw (z), \hskip 11 pt \text{ if  } g=sz, s\in S, z\in U , \\
%		& 0, \hskip 40 pt \text{ otherwise } .
%		\end{split}\right.
%		\end{equation*}
Clearly, $f$ is well defined and $f \in C_c^{\infty} (G)$ (indeed, $f \in C_c^{\infty} (S \cdot U)$).
   By Corollary \ref{cor53}, along with the  $(\omega_{\pi}, Z)$-equivariance of $j_{\pi, \shskip \psi}$ (see \eqref{4eq: psi variant}), we have
   \begin{align*}
    J_{\pi,\shskip \psi}(  f)&=\int_G  f(g) j_{\pi,\shskip \psi}(g) dg  =    \int_{U} \tw(z)d  z \int_Sh(s) j_{\pi,\shskip \psi}(s)ds     = \int_S h(s) j_{\pi,\shskip \psi}(s)ds.
   \end{align*}
   On the other hand, since $\sigma$ is the restriction of $\pi$ to $S$,    we may prove  in a similar fashion that $\pi (f) \varv = \sigma (h) \varv$ for any $\varv \in H$. Precisely,
   \begin{align*}
   \int_G f(g)\pi(g)\varv \,  dg   = \int_{U } \tw(z)d  z \int_S h(s) \pi(s)\varv \,  ds
   =  \int_S h(s) \sigma (s)\varv \,  ds.
   \end{align*}
   In view of the definitions in \S \ref{sec: distributions}, it   follows that $\varv_{L, \shskip f}=\varv_{L,\shskip h}$ and that $J_{\pi,\shskip \psi}(f)=J_{\sigma,\shskip \psi}(h)$. The proof is now completed.
   \end{proof}

For $h \in C_c^{\infty} (S)$ we introduced the orbital integral
\begin{align}\label{eq: O NN (g)}
O^{N, \shskip N}_{h, \shskip \psi} (g) = \iint      h (n(x) g   n(y) ) \psi (x) \psi (y) d x d y.
\end{align}
This orbital integral was studied by Jacquet. In particular, it is proven in \cite{Jacquet-RTF}
that the integral in \eqref{eq: O NN (g)} converges absolutely for all $g \in  X \cap S = N (A\cap S) \varw N$. In view of Proposition \ref{prop54: regularity of J on S}, along with \eqref{4eq: psi variant, SL}, for $h \in C_c^{\infty} (S)$ we have
\begin{align}\label{eq: J = int O}
J_{\sigma,\shskip \psi} (h) = \int  O^{N, \shskip N}_{h, \shskip \psi} (s(a) \varw) j_{\sigma,\shskip \psi} (s(a) \varw) \|a\|^{-2} \dx  a.
\end{align}

	\section{Relative Bessel functions for  $\GL_2 (\BC)$}\label{sec: relative Bessel functions}
	
	In this section,  we prove that    relative Bessel distributions  can be represented by  real analytic   functions on $U = A (N \smallsetminus \{0\} ) \varw_0 N$. This follows directly from (the distributional version of) an explicit formula in \cite{Qi-II-G} (see \eqref{eq: main 1} and \eqref{6eq: Fourier} below) for the Fourier transform of the Bessel function $\bfJ_{ \mu,\shskip   m} \lp z \rp$.
	
	Unlike \cite{BaruchMao-Real}, the distributional integral formula in \cite{Qi-II-G} would enable us to prove the regularity of  relative Bessel distributions on the whole open Bruhat cell $ X = A N \varw_0 N$ rather than its open dense subset $U$. Furthermore, with this observation,  our proof of  the  full regularity of   relative Bessel distributions on $G$ becomes tremendously easier compare to the approach in \cite{BaruchMao-Real}. See \S \ref{sec: regularity relative} for the details.
	
	\subsection{A formula for the Fourier transform of Bessel functions} The Fourier transform $\widehat f  $ of a Schwartz  function $f$ on $\BC$ is defined by
	\begin{equation*}%\label{1eq: Fourier, C}
	\widehat f (u) = \sideset{ }{_\BC }{\iint} \shskip  f (z) e(- \Tr  (uz) ) \shskip  i d z   \nwedge   d \overline z.
	\end{equation*}
%	with $\Tr (z) = z + \overline z$.
According to \cite[Corollary 1.5]{Qi-II-G}, when $|\Re \mu | < \frac 1 2$ and $m$ is even, we have
	\begin{equation}\label{6eq: Fourier}
	\begin{split}
	\sideset{ }{_{\BC } }{\iint}  \hskip - 2 pt   \bfJ_{\mu,\shskip  m} \lp z \rp    \widehat f (z)  \frac {i d z   \nwedge   d \overline z} {  {|z|}}   = \frac  1 2 \sideset{ }{_{\BC } }{\iint}     e \lp \Tr  \lp \frac 1 {2 u} \rp \rp \hskip - 2 pt \bfJ_{\frac 1 2 \mu,\shskip  \frac 1 2 m} \lp  \frac 1  { 16 u ^{ 2} } \rp \hskip - 1 pt f (u) \frac {i d u  \nwedge  d \overline u} {  {|u|}},
	\end{split}
	\end{equation}
	if $f$ is a Schwartz function on $\BC$. Note that there is abuse of the notation $d z$, as $d z$ in this article denotes $2 |\lambdaup|$ times of the Lebesgue measure on $\BC$.
	
	It is critical that  the test function $f$ in \eqref{6eq: Fourier transform of j} only need to be Schwartz on $\BC$ and rapid decay or vanishing at $0$ is not required. We also remark that the deduction from  \eqref{eq: main 1} to \eqref{6eq: Fourier} is   not so straightforward; see \cite[\S 6]{Qi-II-G} for more details.
	
	Recall from \eqref{def: Fourier} that the  Fourier transform with respect to the additive character $\psi (z) = \psi_{\lambdaup} (z) = e(\Tr(\lambdaup z))$ is defined by
	\begin{align*}
	\widehat f (u) = \int_{ \BC} f (z) \psi  (uz) d z.
	\end{align*}
	Thus, in view of \eqref{4eq : j = J}, %\eqref{4eq: j = J, SL} and $ s (1/z) \varw = \varw s (z) $,
	the identity \eqref{6eq: Fourier} may be rephrased as
	\begin{align} \label{6eq: Fourier transform of j}
	\sideset{ }{_{\BCx} }{\int}   j_{\pi, \shskip \psi} ( t(a) \varw_0 )    \widehat f (a)   \dx  a
	=     \sideset{ }{_{\BCx} }{\int}       e \lp \Tr ({\lambdaup} / {2 x} )  \rp \left| \lambdaup / 2 x \right|  \bfJ_{\frac 1 2 \mu,\shskip  \frac 1 2 m} \lp \lambdaup^2 / 16 x^2 \rp  f (x ) {  d x   }  ,
	\end{align}
	if $f$ is a Schwartz function on $\BC$.
	
%	\red{
%	\begin{align*}
%	=    {8 i^{m}}   \sideset{ }{_{\BCx} }{\int}       \psi \lp        {1} / {2 u}   \rp |u|  j_{\sigma, \shskip \psi} ( \varw s (4 u)  )  f (u ) {  d u   }.
%	\end{align*}}
	%$\pi = \pi_{\mu, \shskip  m}$ and $\sigma = \pi_{\frac 1 2 \mu, \shskip  \frac 1 2 m}$.
	
	\subsection{Relative Bessel functions for  $\GL_2 (\BC)$}
	Let $\pi$ be an infinite-dimensional irreducible unitary representation of $G = \GL_2 (\BC)$ with trivial central character. So if $\pi = \pi_{\mu, \shskip  m}$ then $m$ is even. Let $J = J_{\pi, \shskip \psi}$ and $I = I_{\pi, \shskip \psi}$ be the normalized Bessel and relative Bessel distributions defined as in \S \ref{sec: distributions}. Recall the formula \eqref{eq: I(f) = int J},
	\begin{align*}
	 I  (f) = \frac 1 {L (\pi, 1/2)} \int_{\BCx} J  \lp \rho_{l} (t(b)) f \rp \dx  b.
	\end{align*}
	Let $X = A N \varw_0 N$. For $f \in C_c^{\infty} (X)$,   since $ J  $ is represented by the Bessel function $j = j_{\pi, \shskip \psi} $ on $X$ (Corollary \ref{cor53}), it follows  that
	\begin{align}\label{6eq: I = int}
	I  (f) = \frac 1 {L (\pi, 1/2)} \int_{\BCx} \lp \int_{X}  f(t(b)g) j  (g) dg \rp \dx  b .
	\end{align}
	
	\begin{lem}\label{6lem: rel Bessel for split f}
		Let $f \in C_c^{\infty} (X)$ be of the form
		\begin{align}\label{6eq: f splits}
		f (t(a) z(c) n(x) \varw_0 n(y)) = f_1 (a) f_2 (c) f_3 (x) f_4 (y),
		\end{align}
		with $f_1,  f_2 \in  C_c^{\infty}  (\BCx)$ and $f_3, f_4 \in C_c^{\infty}  (\BC)$. Then
		\begin{equation}\label{6eq: double int for I}
		\begin{split}
		& \int_{\BCx}    \int_{G}  f(t(b)g) j  (g) dg \shskip   \dx  b \\
		 = \ & \int_{\BCx} f_1 (b) \dx  b \int_{\BCx} f_2 (c) \dx  c  \int_{\BC}   f_4 (y ) \psi (y) d y \int_{\BCx} j  (t(a) \varw_0) \widehat f_3 (a) \dx  a,
		\end{split}
		\end{equation}
		where $ \widehat f_2 $ is the $\psi$-Fourier transform of $f_2$   defined as in  \eqref{def: Fourier}.
	\end{lem}

\begin{proof}
	Let $dg = \dx  a  \hskip 1 pt  \dx c \hskip 1 pt  d x \hskip 1 pt  d y$ be a Haar measure on $X$ for $g = t(a)  z(c) n(x) \varw_0 n(y)$. In view of \eqref{4eq: psi variant} and \eqref{6eq: f splits},
	the integral on the left hand side of \eqref{6eq: double int for I}   splits into the product
	\begin{align*}
	& \int  f_2 (c) \dx  c  \int   f_4 (y ) \psi (y) d y \int  \lp  \int  f_1 (ab) j  (t(a) \varw_0) \int f_3 (x) \psi (  a x) d x \shskip  \dx  a \rp \dx  b \\
	= \ & \int  f_2 (c) \dx  c  \int   f_4 (y ) \psi (y) d y \iint        f_1 (ab) j  (t(a) \varw_0) \widehat f_3 (a)    \dx  a \shskip   \dx  b.
	\end{align*}
	We claim that the last double integral converges absolutely. Hence we may change the order of integrations and the variable $b$ to $b/a$, getting
	\begin{align*}
	\iint        f_1 (ab) j  (t(a) \varw_0) \widehat f_3 (a)    \dx  a \shskip   \dx  b & = \int         j  (t(a) \varw_0) \widehat f_3 (a)  \int f_1 (ab) \dx  b \shskip   \dx  a \\
	& = \int f_1 (b  ) \dx  b \int    j  (t(a) \varw_0) \widehat f_3 (a)     \dx  a.
	\end{align*}
	We now show   the absolute  convergence. Choose $|\Re \mu | < \rho <  \frac 1 2$.  Lemma \ref{4lem: bounds for J} implies that $j  (t(a) \varw_0)$ may be bounded by $|a|^{1 - 2\rho} +   {|a|}^{\frac 1 2}$. We can find positive constants $A$ and $B$ such that $f_1 (a ) \lll \Upsilon_{A, B} (a)$, with $ \Upsilon_{A, B} (a)$ defined to be the characteristic function on the annulus $\left\{ a : A \leqslant |a| \leqslant B \right\}$. Hence
	\begin{align*}
	\int      \big|  f_1 (ab) j  (t(a) \varw_0) \widehat f_3(a)  \big|  \dx  a  \lll \int \Upsilon_{A/|b|, B/|b|} (a) \big|\widehat f_3 (a) \big| \big( |a|^{- 1- 2 \rho} + |a|^{- \frac 3 2} \big) d a.
	\end{align*}
	Since $\widehat f_3 (a)$ is rapidly decreasing when $|a|$ is large, it follows that the value of the integral above is rapidly decreasing when $|b|$ is small. When $|b|$ is large the integral is bounded by $  |b|^{2 \rho - 1} + |b|^{- \frac 1 2}$. It is then clear that integrating further with $\dx  b = d b / |b|^2$ will converge absolutely  at both $0$ and $\infty$.
\end{proof}

Combining \eqref{6eq: Fourier transform of j}, \eqref{6eq: I = int} and \eqref{6eq: double int for I}, it follows that if we define a function $i = i_{\pi, \shskip \psi}$ supported on $U$ such that
\begin{align}\label{6eq: defn of i}
i_{\pi, \shskip \psi} ( n(x)  \varw_0  ) =     e \lp \Tr ({\lambdaup} / {2 x} )  \rp \left| \lambdaup / 2 x \right|  \bfJ_{\frac 1 2 \mu,\shskip  \frac 1 2 m} \lp \lambdaup^2 / 16 x^2 \rp / {L (\pi, 1/2)},
\end{align}
and that $i_{\pi, \shskip \psi}$ is   right $  (\psi, N)$-equivariant and left $A$-invariant, namely,
\begin{align}\label{eq: equivariance of i}
i_{\pi, \shskip \psi} (t(a) z(c) n(x) \varw_0 n(y)) = \psi (y) i_{\pi, \shskip \psi} ( n(x)  \varw_0  ),
\end{align}
then
\begin{align}\label{6eq: I = int of i}
I_{\pi, \shskip \psi} (f) =  \int_{U} f (g)  i_{\pi, \shskip \psi} (g) d g,
\end{align}
for all $f $ of the form  \eqref{6eq: f splits} in Lemma \ref{6lem: rel Bessel for split f}. Since such functions  span a dense subspace of $ C_c^{\infty} (X)$ (the Stone-Weierstra\ss \hskip 2.5 pt theorem), \eqref{6eq: I = int of i} is actually valid for all $f \in C_c^{\infty} (X)$.

\section{Regularity of relative Bessel distributions over $\GL_2 (\BC)$}\label{sec: regularity relative}
In this section, we prove that the relative Bessel distribution $I = I_{\pi, \shskip \psi}$ is given by the integration against the relative Bessel function $i = i_{\pi,\shskip \psi}$ on the full group $G$.

Since we have already proven in \S \ref{sec: relative Bessel functions} the regularity of $I$ on $X$,  Proposition 2.10 in \cite{Shalika-MO} is now directly applicable to deduce its full regularity on $G$. In particular, we may avoid the introduction of   differential operators and almost all the arguments in \S 5 of \cite{BaruchMao-Real}. This simplification of course  works in the real context as in \cite{BaruchMao-Real}.

We begin with the local integrability of $i $ on $G$. Recall that  $i $ is set to be zero outside $  U =  A (N \smallsetminus \{0\} ) \varw_0 N $.

\begin{prop}\label{prop81: local integ. of i} The relative Bessel function
	$i $ is local integrable on $G$, namely, for any $f\in C_c^\infty(G)$,
	\[
	\int_G |f(g) i (g)|dg< \infty.
	\]
\end{prop}
\begin{proof} The proof is similar to that of Proposition \ref{thm: local integrability}. We have
	\begin{align*}
	& \int_G |f(g)i (g)|dg=\int_U |f(g)i (g)|dg  \\
	%=\ & \iiiint |i(s(a)z(c)n(x)\varw_0n(y))||f(s(a)z(c)n(x)\varw_0n(y))|\dx  a\dx  cdydx   \\
	= & \int |i(n(x)\varw_0)|\left(\iiint|f(s(a)z(c)n(x)\varw_0n(y))|\dx  a\shskip \dx  c\shskip dy\right)dx \\
	= & \int |i(n(x)\varw_0)| \int M (x, c, |f|)  \dx  c \shskip  dx,
	\end{align*}
	in which $M (x, c, |f|) $ is the orbital integral defined as in \eqref{def: M(x,c,f)}.
	Applying Proposition \ref{propA2}, the inner integral over $c$ is zero for $|x|$   large and of the order $1/|x|^\varepsilon$ for small values of $|x|$. On the other hand, by the expression of $i=i_{\pi,\shskip \psi}$ in \eqref{6eq: defn of i} and Lemma \ref{4lem: bounds for J}, we know that $i (n(x)\varw_0)$ is bounded when $|x|$ is small. It is now clear that the integral above is convergent.
\end{proof}

By Proposition \ref{prop81: local integ. of i}, we can define the distribution $\widetilde I=\widetilde I_{\pi,\shskip \psi}$ by
\[
\widetilde I (f)=\int_G f(g) i (g) dg.
\]

\begin{prop}\label{prop82:regularity of I}
	For any $f\in C_c^\infty(G)$, we have
	\[
	I (f)=\widetilde I (f).
	\]
\end{prop}
\begin{proof}
	Let $T$ be the distribution on $G$ defined by the difference $T=I-\widetilde I$. First, in view of \eqref{6eq: I = int of i},  $I$ and $\widetilde I$ coincide on $X$, so $T$ is supported on the Borel $B = G\smallsetminus X$. Second, it is clear that $T$ satisfies  $\rho_r(n(y))I=\psi(y)I$. Third, we claim that $T$ is an eigen-distribution of  the Casimir element $\varDelta$ in the universal enveloping algebra of the Lie algebra of $G$. To see this, we first note that  $I$ is an eigen-distribution of $\varDelta$ according to \cite[\S 3]{Shalika-MO}, that is,  $\varDelta I= \kappa I$ ($\kappa \in \BC$). As aforementioned, $I$ is represented by the (real analytic) function $i$ when restricted on $ U $, so $i$ is an eigen-function of $\varDelta$  on $U$. % (as a  real analytic function).
	 Precisely, we have $\varDelta \shskip i = \kappa \shskip i$, which further implies $\varDelta \widetilde I= \kappa \widetilde I$. The third claim is now proven. Under the three conditions above, by    \cite[Proposition 2.10]{Shalika-MO}, we must have $T=0$.
\end{proof}

Finally, for $f \in C_c^{\infty} (G)$ define the orbital integral
\begin{align}\label{eq: O AN (g)}
O^{A, \shskip N}_{f, \shskip \psi} (g) = \int_A \int_N  f (a g  n) \psi (n) \dx a \hskip 1 pt d n.
\end{align}
It is proven in \cite{Jacquet-RTF} that the integral in \eqref{eq: O AN (g)}  converges absolutely for any $g \in U =  A (N \smallsetminus \{0\} ) \varw_0 N$. In view of Proposition \ref{prop81: local integ. of i} and \ref{prop82:regularity of I}, along with \eqref{eq: equivariance of i}, % if we define
%\begin{equation}
%%O_{f, \shskip \psi} (x) = O^{A, \shskip N}_{f, \shskip \psi} (n(x) \varw_0),
%\end{equation}
  for  $f \in C_c^{\infty} (G)$ we have
\begin{align}\label{eq: I = int O}
I_{\pi,\shskip \psi} (f) = \int_{ \BCx} O^{A, \shskip N}_{f, \shskip \psi} (n(x) \varw_0) i_{\pi,\shskip \psi} (n(x) \varw_0) d x.
\end{align}

	\section{Bessel identities over $\BC$}
	
	We are now ready to establish the Waldspurger correspondence between   irreducible unitary representations of $ G = \GL_2 (\BC) $ with   trivial central character (that is, representations of $  G/Z = \PGL_2 (\BC) $) and  irreducible unitary representations of $S = \SL_2 (\BC)$ from the identities between their Bessel functions.
	
	Fix the nontrivial additive character $ \psi = \psi_{\lambdaup}$ of $\BC$ defined by $\psi (z) = e (\Tr (\lambdaup z) )$. Let $D \in \BCx$ and define $ \psi^D (z) = \psi (D z)$. Define the Weil factor $\gamma (z, \psi^D) $ by
	\begin{align*}
	\gamma (z, \psi^D) = 1/{\textstyle \sqrt {\|2 D\|}} .
	\end{align*}
	Note that $ \gamma (z, \psi^D)$ does not depend on $z$, but we would rather keep $z$ in this conventional notation.
	Define a {transfer factor}
	\begin{align}
	\Delta_{D,\shskip \psi} (z) = \gamma (z, \psi^D) \psi (2D/ z) {\textstyle \sqrt {\|z\|} }.
	\end{align}
	Let $\sigma$ be an irreducible unitary representation of $S$.
	When changing $\psi = \psi_{\lambdaup}$ to $\psi^D = \psi_{\lambdaup D}$, we wish to keep the Haar measure fixed on $S$ with respect to $\psi  $ and instead re-normalize the formulae of $ j_{\sigma,\shskip \psi^D } $ in \S \ref{sec: Bessel, GL2} by an extra factor $   \sqrt {\|D\|} $\footnote{In the paper \cite{BaruchMao-Real}, the more reasonable re-normalizing factor should be $\sqrt {|D|}$ not $1/\sqrt {|D|}$, so their formulae (19.2) and (19.4) need small modifications on the factors involving $|D|$.}.

	\begin{defn}
	Let $\pi$ be an irreducible unitary representation of $G / Z$. We say that an irreducible unitary representation $\sigma$ of $S$ corresponds to $\pi$ if the following equality {\rm(}Bessel identity{\rm)} holds
	\begin{align}\label{eq: Bessel indentity}
	 i_{\pi, \shskip \psi} (n (z/ 4 D)\varw_0) = \frac {4 \Delta_{D,\shskip \psi} (z) \epsilon (\pi, 1/2, \psi)   } {L(\pi, 1/2)} j_{\sigma, \shskip \psi^D} (\varw s (z)),
	\end{align}
	for all $z \in \BCx$.
	\end{defn}

The following theorem is the main theorem of this paper.

\begin{thm}\label{thm: main}
	For each irreducible   unitary representation $\pi$ of $G / Z$, there exists a corresponding irreducible   unitary representation $\sigma $ of $S$ satisfying the Bessel identity \eqref{eq: Bessel indentity}. The correspondence is given by
	\begin{align*}
	\pi_{\mu, \shskip m}  \longrightarrow \sigma_{\frac 1 2 \mu, \shskip \frac 1 2 m}, \hskip 10 pt \text{for } \mu \in i \shskip \BR, m \text{ even, or } \mu \in \big( 0, \tfrac 1 2\big) \hskip -1 pt , m = 0.
	\end{align*}
\end{thm}

\begin{proof}
	The proof is a simple comparison between the formulae \eqref{4eq: j = J, SL} and \eqref{6eq: defn of i}. To verify the equality,  it should be noted that by \cite[(4.7)]{Knapp} and \cite[\S 3]{Tate} we have $\epsilon (\pi, s, \psi_{\lambdaup}) = i^{|m|} \omega_{\pi} (\lambdaup) \|\lambdaup\|^{2s -1}$ if $\pi = \pi_{\mu, \shskip  m}$, and hence $\epsilon (\pi, 1/2, \psi_{\lambdaup}) = i^{|m|} = (-1)^{\frac 1 2 m}$ when $m$ is even so that $\omega_{\pi} (\lambdaup) = 1$. Note that $\epsilon (\pi, 1/2, \psi)$ is actually independent on $\psi$.
\end{proof}

We remark that $ \sigma = \Theta (\pi) = \Theta (\pi , \psi^D)$ according to the notation of the theta correspondence of Waldspurger \cite{Waldspurger-Shimura}. Unlike the real case (see \cite[Theorem 19.2]{BaruchMao-Real}),  here $\sigma$    is independent on the additive character $\psi^D$. %Thus the theorem is equivalent to the following: the identity \eqref{eq: Bessel indentity} holds if $\sigma = \Theta (\pi)$.

Finally, we would like to prove an identity in the level of distributions. %We start with introducing two orbital integral studied by Jacquet in \cite{Jacquet-RTF}.
It is proven in \cite{Jacquet-RTF} that for each $f \in C_c^\infty (G)$ there exists $f' \in C_c^\infty (S)$ such that
\begin{equation}\label{eq: O AN = O NN}
O^{A, \shskip N}_{f, \shskip \psi} (n(z/4D) \varw_0) = O^{N, \shskip N}_{f', \shskip \psi^D} (\varw s(z)) \psi (-2D/z) {\textstyle \sqrt {\|z\|}}/ \gamma (z, \psi^D),
\end{equation}
for all $z \in \BCx$. See \eqref{eq: O NN (g)} and \eqref{eq: O AN (g)} for the definitions of these orbital integrals (the Haar measure in \eqref{eq: O NN (g)} however is now chosen with respect to $\psi$ rather than $\psi^D$).   By \eqref{eq: J = int O}, \eqref{eq: I = int O} and Theorem \ref{thm: main} we have the following theorem.

\begin{thm}
Assume that $f$ and $f '$ satisfy \eqref{eq: O AN = O NN} and that $\pi$ and $\sigma$ correspond as in Theorem {\rm\ref{thm: main}}. Then
\begin{equation}
I_{\pi,\shskip \psi} (f) = J_{\sigma,\shskip \psi^D} (f') \epsilon (\pi, 1/2, \psi)    / \|2 D\| L (\pi, 1/2).
\end{equation}
\end{thm}

	\appendix
	
	\section{Orbital integrals}
	
	 In this appendix, we state some preliminary analytic results on the $(N,N)$ and $(A,N)$ orbital integrals that were needed for verifying the absolute convergence of certain integrals. While the volumn estimates are slightly more complicated (they can still be done in the polar coordinates without much difficulty), the proofs are literally identical with those for the real case in \cite[\S 4.1, 4.2]{BaruchMao-Real} and will be omitted here.
	
   \subsection{$(N,N)$ orbital integrals}
   For $f\in C_c (G)$, $a,c\in \BCx$, we define the following orbital integral
  \begin{equation}\label{def: J (a,c,f)}
   J(a,c,f)=\iint f(n(x)z(c)s(a)\varw_0n(y))dxdy.
  \end{equation}

\begin{prop}\label{prop:A1}
Let $f\in C_c (G)$. Then

{\rm(1).} $J(a,c,f)$ converges absolutely for all $a,c\in \BCx$.

{\rm(2).} $J(a,c,f)$ is compactly supported as a function of $c$ in $\BCx$, independent on $a$.

{\rm(3).} $J(a,c,f)$ is zero when $|a|$ is small, independent on $c$.

{\rm(4).} $ J(a,c,f) =O(|a|^{2+\varepsilon})$ for any $\varepsilon>0$, when $|a|$ is large, independent on $c$.
\end{prop}
%\begin{proof}
%The proofs of (1), (2), (3) are the same as those of Proposition 4.1 in \cite{BaruchMao-Real}, which will be omitted. For (4), we extend the function to $\BC^4$ by setting it to be zero outside $G$. Take $xca^{-1}\to x$ and $yca^{-1}\to y$, we have
%\[
%J(a,c,f)=|a/c|^2\int f\begin{pmatrix} x&  {xya}/{c}+ac \\  {c}/{a} & y \end{pmatrix}dxdy.
%\]
%Since $f$ is compactly supported, the above integral is bounded by some constant times the area of a region in $\BC^2$ of the form $\{(x,y):|x|\le X, |y|\le Y, |xya+ac^2|\le Z \}$ for positive constants $X, Y, Z$. Direct computation shows that the area of such region is of the order of $|a|^{-1}(\log |a|)^3$, thus (4) is proved.
%\end{proof}

\subsection{$(A,N)$ orbital integrals}
 For $f\in C_c (G), x\in \BC, c\in \BCx$, define orbital integral
\begin{equation}\label{def: M(x,c,f)}
 M(x,c,f)=\iint f(s(a)z(c)n(x)\varw_0n(y))\dx  a dy.
\end{equation}
%We list some   properties of $M(x,c,f)$  in the following proposition.

\begin{prop}\label{propA2}
Let $f\in C_c (G)$. Then

{\rm(1).} $M(x,c,f)$ converges absolutely for all $c, x \in \BCx$.

{\rm(2).} $M(x,c,f) $ is compactly supported as a function of $c$ in $\BCx$, independent on $x$.

{\rm(3).} $M(x,c,f) $ is zero for large values of $|x|$, independent on $c$.

{\rm(4).} $M(x,c,f) =O(1/|x|^{\varepsilon})$ for any $\varepsilon>0$, when $|x|$ is small, independent on $c$. %The implies constant is also independent on $c$ {\rm(}but dependent on $\varepsilon${\rm)}.
\end{prop}

%	\bibliographystyle{alphanum}
%	\bibliography{references}

\def\cprime{$'$} \def\cprime{$'$}

\end{document}